\documentclass{article}

\usepackage{latexsym}
\usepackage{amsmath}
\usepackage{amssymb,amsfonts,amscd}
\usepackage{graphicx}
\usepackage{multirow}
\usepackage{hyperref}
\usepackage{verbatim}

\newcommand\qed{{\hspace*{\fill}$\Box$\vskip12pt plus 1pt}}

\newtheorem{theorem}{Theorem}[section]
\newtheorem{proposition}[theorem]{Proposition}
\newtheorem{corollary}[theorem]{Corollary}
\newtheorem{lemma}[theorem]{Lemma}
\newtheorem{rem}[theorem]{Remark}
\newtheorem{defn}[theorem]{Definition}
\newtheorem{ex}[theorem]{Example}
\newtheorem{alg}[theorem]{Algorithm}

\newenvironment{remark}{\begin{rem}\em}{\end{rem}}
\newenvironment{definition}{\begin{defn}\em}{\end{defn}}

\newenvironment{proof}{{\noindent\bf Proof.\ }}{\qed}

\newcommand{\bone}{{\bf 1}}

\newcommand{\bN}{{\mathbb N}}
\newcommand{\bC}{{\mathbb C}}

\newcommand{\bR}{{\mathbb R}}

\newcommand{\bZ}{{\mathbb Z}}

\newcommand\sF{{\mathcal F}}
\newcommand\sG{{\mathcal G}}

\newcommand\sM{{\mathcal M}}

\newcommand\sP{{\mathcal P}}
\newcommand\sQ{{\mathcal Q}}

\newcommand\sV{{\mathcal V}}

\newcommand\Var{\sV}

\renewcommand{\d}{\partial}

\newcommand{\alphaCert}{{\tt alphaCertified}}
\newcommand{\alphaCertS}{{\tt alphaCertified}~}

\newcommand{\BertiniS}{{\tt Bertini}~}
\newcommand\AlgCertX{{\bf CertifySoln}}

\newcommand{\mytilde}{\raise.17ex\hbox{$\scriptstyle\mathtt{\sim}$}}

\begin{document}





\title{Certifying solutions to square systems of polynomial-exponential equations}

\author{Jonathan D. Hauenstein\thanks{Department of Applied and Computational Mathematics and Statistics, University of Notre Dame, Notre Dame, IN 46556 (hauenstein@nd.edu,
\url{www.nd.edu/\~jhauenst}).  This author was partially supported by 
NSF grants DMS-0915211, DMS-1114336, and ACI-1460032.}
\and Viktor Levandovskyy\thanks{Lehrstuhl D f\"ur Mathematik, RWTH Aachen, Templergraben 64, 52062 Aachen Germany
(Viktor.Levandovskyy@math.rwth-aachen.de, \url{www.math.rwth-aachen.de/\~Viktor.Levandovskyy}).}}

\date{June 30, 2015}

\maketitle

\begin{abstract}

\noindent Smale's $\alpha$-theory certifies that Newton iterations will converge 
quadratically to a solution of a square system of analytic functions based on the Newton residual and
all higher order derivatives at the given point.  
Shub and Smale presented a bound for the higher order derivatives of a 
system of polynomial equations based in part on the degrees of the equations.
For a given system of polynomial-exponential equations, we consider
a related system of polynomial-exponential equations and provide a
bound on the higher order derivatives of this related system.
This bound yields a complete algorithm for certifying solutions to 
polynomial-exponential systems, which is implemented in \alphaCert.
Examples are presented to demonstrate this certification algorithm.



\end{abstract}
\normalsize

\section{Introduction}\label{Sec:Intro}

A map $f:\bC^n\rightarrow\bC^n$ is called a square system of
polynomial-exponential functions if $f$ is polynomial in both the variables
$x_1,\dots,x_n$ and finitely many exponentials
of the form $e^{\beta x_i}$ where $\beta\in\bC$.
That is, there exists a polynomial system $P:\bC^{n+m}\rightarrow\bC^n$,
analytic functions $g_1,\dots,g_m:\bC\rightarrow\bC$,
and integers $\sigma_1,\dots,\sigma_m\in\{1,\dots,n\}$ such that
\[
f(x_1,\dots,x_n) = P(x_1,\dots,x_n,g_1(x_{\sigma_1}),\dots,g_m(x_{\sigma_m}))
\]
where each $g_i$ satisfies some linear homogeneous partial differential equation (PDE)
with complex coefficients.  In particular, for each $i = 1,\dots,m$,
there exists a positive integer $r_i$ and a linear function $\ell_i:\bC^{r_i+1}\rightarrow\bC$
such that $\ell_i(g_i,g_i',\dots,g_i^{(r_i)}) = 0$.

Consider the square polynomial-exponential system $\sF:\bC^{n+m}\rightarrow\bC^{n+m}$ where
\begin{equation}\label{Eq:ReductionF}
\sF(x_1,\dots,x_n,y_1,\dots,y_m) = \left[\begin{array}{c}
P(x_1,\dots,x_n,y_1,\dots,y_m) \\
y_1 - g_1(x_{\sigma_1}) \\ \vdots \\ y_m - g_m(x_{\sigma_m}) \end{array}\right].
\end{equation}
Since the projection map $(x,y)\mapsto x$ defines a bijection between the solutions
of $\sF(x,y) = 0$ and $f(x) = 0$, we will only
consider certifying solutions to square systems of polynomial-exponential equations
of the form $\sF(x,y) = 0$.

For a square system $g:\bC^n\rightarrow\bC^n$ of analytic functions,
a point $x\in\bC^n$ is an {\em approximate solution} of $g = 0$
if Newton iterations applied to $x$ with respect to $g$
quadratically converge immediately to a solution of $g = 0$.
The certificate returned by our approach that a point is an
approximation solution of $\sF = 0$ is an $\alpha$-theoretic certificate.
In short, $\alpha$-theory, which started for systems of analytic equations in \cite{S86},
provides a rigorous mathematical foundation for the fact that if the
Newton residual at the point is small and the higher order derivatives
at the point are controlled, then the point is an approximate solution.
For polynomial systems, by exploiting the fact that there are only finitely
many nonzero derivatives, Shub and Smale \cite{SS93} provide a bound on all
of the higher order derivatives.
For polynomial-exponential systems, our approach uses the structure
of $\sF$ together with the linear functions $\ell_i$
to bound the higher~order~derivatives.

Systems of polynomial-exponential functions naturally arise in many applications
including engineering, mathematical physics, and control theory, to name a few.
On the other hand, such functions are typical solutions to systems of linear partial
differential equations with constant coefficients.
Systems, including ubiquitous functions like $\sin(x)$, $\cos(x)$, $\sinh(x)$, and $\cosh(x)$, can be
equivalently reformulated as systems of polynomial-exponential functions, since
these functions can be expressed as polynomials involving $e^{\beta x}$ for suitable $\beta\in\bC$.
Since computing all solutions to such systems is often nontrivial,
methods for approximating and certifying some solutions for general systems
is very important, especially in the aforementioned applications.

In the rest of this section, we introduce the needed concepts from
$\alpha$-theory.  Section~\ref{Sec:CertifySolns} formulates
the bounds for the higher order derivatives of polynomial-exponential
systems and presents a certification algorithm for polynomial-exponential systems.
In Section~\ref{Sec:ApproxSolns}, we discuss methods for 
generating numerical approximations to solutions of polynomial-exponential systems.
Section~\ref{Sec:Examples} describes the implementation of the certification
algorithm in \alphaCertS as well as demonstrating the algorithms on a collection of examples.
Appendix~\ref{Sec:Appendix} demonstrates
the input, command-line execution, and 
output of \alphaCertS for a polynomial-exponential system
from Section~\ref{Sec:RRdyad}.  Files for all of the examples are
available at \url{www.nd.edu/~jhauenst/PolyExp}.

\subsection{\texorpdfstring{Smale's $\alpha$-theory}{Smale's alpha-theory}}\label{Sec:alphaTheory}

We provide a summary of the elements of $\alpha$-theory used in
the remainder of the article as well as in \alphaCert.
Hence, this section closely follows \cite[\S~1]{alphaCertifiedPaper}
expect ``polynomial'' is replaced by ``analytic.''
We focus on {\em square} systems, which are systems with
the same number of variables and functions,
with more details provided in \cite{BCSS}.

Let $f:\bC^n\rightarrow\bC^n$ be a system of analytic functions
with zeros $\Var(f) = \{\xi\in\bC^n~|~f(\xi) = 0\}$
and $Df(x)$ be the Jacobian matrix of $f$ at $x$.
For a point $x\in\bC^n$, the point $N_f(x)$ is called
the {\em Newton iteration of $f$ at $x$} where
the map $N_f:\bC^n\rightarrow\bC^n$ is defined by
\[
N_f(x) = \left\{\begin{array}{ll} x - Df(x)^{-1}f(x) & \hbox{if $Df(x)$ is invertible,} \\
 x & \hbox{otherwise.} \end{array}\right.
\]
For $k\in\bN$, let $N_f^k(x)$ be the $k^{th}$ Newton iteration of $f$ at $x$, that is,
\[
 N_f^k(x) = \underbrace{N_f \circ \cdots \circ N_f}_{k \hbox{\scriptsize ~times}}(x).
\]
The following defines an approximate solution of $f$ to be a point which converges quadratically
in the standard Euclidean norm on $\bC^n$ to a point in $\Var(f)$.

\begin{definition}\label{Def:ApproxSoln}
Let $f:\bC^n\rightarrow\bC^n$ be an analytic system.  A point $x\in\bC^n$
is an {\em approximate solution} of $f = 0$ with {\em associated solution} $\xi\in\Var(f)$
if, for every $k\in\bN$,
\[
\|N_f^k(x) - \xi\| \leq \left(\frac{1}{2}\right)^{2^k-1} \|x - \xi\|.
\]
\end{definition}

Clearly, every solution of $f = 0$ is an approximate solution of $f = 0$.
Additionally, when $Df(x)$ is not invertible, then a point $x$ is
an approximate solution of $f = 0$ if and only if $x\in\Var(f)$.
When $Df(x)$ is invertible, the results of $\alpha$-theory
provide a certificate that $x$ is an approximate solution of $f = 0$.
This certificate is based on $\alpha(f,x)$, $\beta(f,x)$, and $\gamma(f,x)$, namely
\begin{eqnarray}
\alpha(f,x) &=& \beta(f,x)\cdot\gamma(f,x), \nonumber\\
\beta(f,x) &=& \|x - N_f(x)\| = \|Df(x)^{-1} f(x)\|, \hbox{~and} \nonumber \\
\gamma(f,x) &=& \displaystyle\sup_{k\geq2} \left\|\frac{Df(x)^{-1} D^k f(x)}{k!}\right\|^{\frac{1}{k-1}} \label{Eq:Gamma}
\end{eqnarray}
where $D^k f(x)$ is the $k^{th}$ derivative of $f$ (see \cite[Chap.~5]{L83}).

When $Df(x)$ is not invertible, we define $\beta(f,x)$ as zero and $\gamma(f,x)$ as infinity.
The constant $\alpha(f,x)$ is then the indeterminate form $0\cdot\infty$
which is defined based on the value of $f(x)$.
If $f(x) = 0$, then $\alpha(f,x)$ is defined as zero, otherwise
$\alpha(f,x)$ is defined as infinity.

The following lemma, which is a conclusion of Theorem 2 of \cite[Chap.~8]{BCSS},
shows that, when~$x$ is an approximate solution of $f = 0$,
the distance between $x$ and its associated solution can be bounded in terms of $\beta(f,x)$.
Moreover, this bound can be used to produce a certificate that two approximate solutions have
distinct associated solutions.

\begin{lemma}\label{Lemma:DistBound}
Let $f:\bC^n\rightarrow\bC^n$ be an analytic system.
If $x\in\bC^n$ is an approximate solution of $f = 0$
with associated solution $\xi$, then
\[
\|x - \xi\| \leq 2\beta(f,x).
\]
Moreover, if $x_1,x_2\in\bC^n$ are approximate solutions of
$f = 0$ with associated solutions $\xi_1,\xi_2$, respectively, then
$\xi_1 \neq \xi_2$ provided that
\[
\|x_1 - x_2\| > 2(\beta(f,x_1) + \beta(f,x_2)).
\]
\end{lemma}
\begin{proof}
Both results immediately follow from the triangle inequality.  In particular,
\[
\|x-\xi\| \leq \|x - N_f(x)\| + \|N_f(x) - \xi\| \leq \beta(f,x) + \frac{1}{2} \|x - \xi\|
\]
yields $\|x - \xi\| \leq 2 \beta(f,x)$.  Additionally,
\[
\|x_1 - x_2\| \leq \|x_1 - \xi_1\| + \|\xi_1 - \xi_2\| + \|\xi_2 - x_2\| \leq 2(\beta(f,x_1) + \beta(f,x_2)) + \|\xi_1 - \xi_2\|
\]
yields that $\xi_1 \neq \xi_2$ when $\|x_1 - x_2\| > 2(\beta(f,x_1) + \beta(f,x_2))$.
\end{proof}

The following theorem, called an $\alpha$-theorem,
is a version of Theorem 2 of \cite[Chap.~8]{BCSS} which
shows that the value of $\alpha(f,x)$ can be used to
produce a certificate that $x$ is an approximate solution of $f = 0$.

\begin{theorem}\label{Thm:AlphaThm}
If $f:\bC^n\rightarrow\bC^n$ is an analytic system and $x\in\bC^n$ with
\[
\alpha(f,x) < \frac{13 - 3\sqrt{17}}{4} \approx 0.157671,
\]
then $x$ is an approximate solution of $f = 0$.
\end{theorem}

The following theorem, called a robust $\alpha$-theorem that 
is a version of Theorem~4 and Remark~6 of \cite[Chap.~8]{BCSS},
shows that the value of $\alpha(f,x)$ and $\gamma(f,x)$
can be used to produce a certificate that $x$ and another point $y$
have the same associated solution.

\begin{theorem}\label{Thm:RobustAlphaThm}
Let $f:\bC^n\rightarrow\bC^n$ be an analytic system and $x\in\bC^n$ with
$\alpha(f,x) < 0.03$.  If $y\in\bC^n$ such that
\[
\|x - y\| < \frac{1}{20\gamma(f,x)},
\]
then $x$ and $y$ are both approximate solutions of $f = 0$
with the same associated solution.
\end{theorem}

Let $\pi_\bR:\bC^n\rightarrow\bR^n$ be the real projection map
defined by $\pi_\bR(x) = \displaystyle\frac{x + \overline{x}}{2}$
where $\overline{x}$ is the complex conjugate of $x$.
If $f$ is an analytic system such that $N_f(\overline{x}) = \overline{N_f(x)}$
for all $x$ such that $Df(x)$ is invertible, then $N_f$ defines a real map, i.e., $N_f(\bR^n)\subset\bR^n$.
In particular, if $x$ is an approximate solution of $f = 0$ with associated solution $\xi$,
then $\overline{x}$ is also an approximate solution of $f = 0$ with associated solution $\overline{\xi}$
and $\beta(f,x) = \beta(f,\overline{x})$.  The following proposition,
which is a summary of the approach in \cite[\S~2.1]{alphaCertifiedPaper},
can be used to determine if the associated solution of an approximation solution is real.

\begin{proposition}\label{Prop:Real}
Let $f:\bC^n\rightarrow\bC^n$ be a polynomial system such that $N_f(\overline{x}) = \overline{N_f(x)}$
for all $x\in\bC^n$ such that $Df(x)$ is invertible.  Let $x\in\bC^n$ be an approximate solution of $f = 0$ with
associated solution $\xi$.
\begin{enumerate}
\item\label{Item:Real1} If $\|x - \pi_\bR(x)\| > 2\beta(f,x)$, then $\xi\notin\bR^n$.
\item\label{Item:Real2} If $\alpha(f,x) < 0.03$ and $\displaystyle\|x-\pi_\bR(x)\| < \frac{1}{20\gamma(f,x)}$, then $\xi\in\bR^n$.
\end{enumerate}
\end{proposition}
\begin{proof}
Since $\|x - \overline{x}\| = 2\|x - \pi_\bR(x)\|$ and $\beta(f,x) = \beta(f,\overline{x})$,
Item~\ref{Item:Real1} follows by concluding $\xi \neq \overline{\xi}$ using
Lemma~\ref{Lemma:DistBound}.  Item~\ref{Item:Real2} follows from Theorem~\ref{Thm:RobustAlphaThm}
together with $\pi_\bR(x)\in\bR^n$ and $N_f(\bR^n)\subset\bR^n$.
\end{proof}

\subsection{Bounding higher order derivatives}\label{Sec:HigherDerivs}

The constant $\gamma(f,x)$ defined in (\ref{Eq:Gamma}) yields
information regarding the higher order derivatives of $f$ evaluated
at $x$.  Even though, for polynomial systems, $\gamma(f,x)$
is actually a maximum of finitely many values, it is often computationally
difficult to compute exactly.  However, in the polynomial case,
it can be bounded above based in part on the degrees of the polynomials \cite{SS93}.
Due to the nature of polynomial-exponential systems, this
bound will be used in our algorithm presented in Section~\ref{Sec:CertifySolns}
for certifying solutions to polynomial-exponential systems.

Let $g:\bC^n\rightarrow\bC$ be a polynomial of degree $d$ where
$g(x) = \sum_{|\rho|\leq d} a_\rho x^\rho$ and
\[
\|g\|^2 = \frac{1}{d!} \sum_{|\rho|\leq d} \rho!\cdot(d - |\rho|)!\cdot|a_\rho|^2
\]
is the standard unitarily invariant norm on the homogenization of $g$.
For a polynomial system $f:\bC^n\rightarrow\bC^n$
with $f(x) = [f_1(x),\dots,f_n(x)]^T$, we have
\[
\|f\|^2 = \sum_{i=1}^n \|f_i\|^2.
\]
For a point $x\in\bC^n$, define $\|x\|_1^2 = 1 + \|x\|^2 = 1 + \sum_{i=1}^n |x_i|^2$.

The following is an affine version of Propositions~1 and~3 from \cite{SS93}.

\begin{proposition}\label{Prop:BoundEval}
If $g:\bC^n\rightarrow\bC$ is a polynomial of degree $d$,
then, for all $x\in\bC^n$ and $k\geq 1$,
\[
|g(x)| \leq \|g\|\cdot\|x\|_1^d \hbox{~~and~~} \|D^kg(x)\| \leq
d\cdot(d-1)\cdots(d-k+1)\cdot\|g\|\cdot\|x\|_1^{d-k}.
\]
\end{proposition}

Let $k\geq 2$.  Lemma~3 of \cite{SS93} yields
\[
\left(\frac{d\cdot(d-1)\cdots(d-k+1)}{d^{1/2}\cdot k!}\right)^{\frac{1}{k-1}} \leq \frac{d^{1/2}(d-1)}{2}
\leq \frac{d^{3/2}}{2}.
\]
Additionally, since $\|x\|_1\geq1$, we know $\|x\|_1^{d-1}\geq\|x\|_1^{d-k}$.
These facts together with Proposition~\ref{Prop:BoundEval} yield
$$\begin{array}{rcl}
\displaystyle\left\|\frac{D^k g(x)}{k!}\right\|^{\frac{1}{k-1}}
 &\leq& \left(\displaystyle \frac{d^{1/2}\cdot\|D^k g(x)\|}{d^{1/2}\cdot k!}\right)^{\frac{1}{k-1}} \\
 &\leq& d^{\frac{1}{2(k-1)}} \left(\displaystyle\frac{d\cdot(d-1)\cdots(d-k+1)\cdot\|g\|\cdot\|x\|_1^{d-k}}
       {d^{1/2} k!}\right)^{\frac{1}{k-1}} \\
 &\leq& \left(d^{1/2}\cdot\|x\|_1^{d-k}\cdot\|g\|\right)^{\frac{1}{k-1}}
    \left(\displaystyle\frac{d\cdot(d-1)\cdots(d-k+1)}{d^{1/2}\cdot k!}\right)^{\frac{1}{k-1}} \\
 &\leq& \displaystyle\frac{d^{3/2}}{2\|x\|_1}
     \left(d^{1/2}\cdot\|x\|_1^{d-1}\cdot\|g\|\right)^{\frac{1}{k-1}}
\end{array}$$
which we summarize in the following proposition.

\begin{proposition}\label{Prop:BoundK}
If $g:\bC^n\rightarrow\bC$ is a polynomial of degree $d$,
then, for all $x\in\bC^n$ and $k\geq 2$,
\[
\displaystyle\left\|\frac{D^k g(x)}{k!}\right\|^{\frac{1}{k-1}} \leq \displaystyle\frac{d^{3/2}}{2\|x\|_1}
     \left(d^{1/2}\cdot\|x\|_1^{d-1}\cdot\|g\|\right)^{\frac{1}{k-1}}.
\]
\end{proposition}

Let $f:\bC^n\rightarrow\bC^n$ be a polynomial system
with $\deg f_i = d_i$.  Define $D = \max d_i$ and
\begin{equation}\label{Eq:MuPoly}
\mu(f,x) = \max\{1,~\|f\|\cdot\|Df(x)^{-1} \Delta_{(d)}(x)\|\}
\end{equation}
assuming $Df(x)$ is invertible where
\begin{equation}\label{Eq:Delta}
\Delta_{(d)}(x) = \left[\begin{array}{ccc} d_1^{1/2}\cdot\|x\|_1^{d_1-1} & & \\ & \ddots & \\ & & d_n^{1/2}\cdot\|x\|_1^{d_n-1} \end{array}\right].
\end{equation}
Since $\mu(f,x)\geq 1$, $\mu(f,x)^{\frac{1}{k-1}} \leq \mu(f,x)$ for any $k\geq2$.

The following version of
Proposition~3 of \cite[\S~I-3]{SS93} yields an upper bound for $\gamma(f,x)$.

\begin{proposition}\label{Prop:GammaBound}
Let $f:\bC^n\rightarrow\bC^n$ be a polynomial system with $\deg f_i = d_i$ and
$D = \max d_i$.  For any $x\in\bC^n$ such that $Df(x)$ is invertible,
\[
\gamma(f,x)\leq\frac{\mu(f,x)\cdot D^{3/2}}{2\cdot\|x\|_1}.
\]
\end{proposition}
\begin{proof}
For $k\geq2$, we have
$$\begin{array}{rcl}
\displaystyle\left\|\frac{Df(x)^{-1} D^k f(x)}{k!}\right\|^{\frac{1}{k-1}}
 &\leq&\left(\|f\|\cdot\|Df(x)^{-1}\Delta_{(d)}(x)\|\right)^{\frac{1}{k-1}}
       \displaystyle\left\|\frac{\Delta_{(d)}(x)^{-1} D^k f(x)}{\|f\|\cdot k!}\right\|^{\frac{1}{k-1}} \\
 &\leq& \mu(f,x)
    \displaystyle\left(\sum_{i=1}^n \frac{\|f_i\|^2}{\|f\|^2}\left(\frac{d_i^{3/2}}{2\cdot\|x\|_1}\right)^{2(k-1)}
      \right)^{\frac{1}{2(k-1)}} \\
 &\leq& \displaystyle\frac{\mu(f,x) D^{3/2}}{2\cdot\|x\|_1}.
 \end{array}$$
\end{proof}

\section{Certifying solutions}\label{Sec:CertifySolns}

Since the bound provided in Proposition~\ref{Prop:GammaBound}
does not apply to a polynomial-exponential system~$\sF$,
we develop a new bound based on the solutions of linear homogeneous
partial differential equations.  With this bound, algorithms
for certifying approximate solutions, distinct associated solutions,
and real associated solutions of \cite{alphaCertifiedPaper}
apply to $\sF$.

Consider $g(x) = e^{\beta x}$ for some $\beta\in\bC$.
Clearly, for any $k\geq 0$, $|g^{(k)}(x)| = |\beta|^k\cdot |g(x)|$.
By letting $B(x) = |g(x)|$ and $C = \max\{1,|\beta| \}$, we have
\begin{equation}\label{Eq:BoundExpBeta}
|g^{(k)}(x)| \leq C^k\cdot B(x).
\end{equation}
The following lemma shows that a similar bound holds in general.

\begin{lemma}\label{Lemma:ExpBoundGen}
Let $c_0,\dots,c_{r-1}\in\bC$, $\ell(x_0,\dots,x_r) = x_r - \sum_{i=0}^{r-1} c_i x_i$,
and $g:\bC\rightarrow\bC$ be an analytic function such that $\ell(g,g',\dots,g^{(r)}) = 0$
and $r$ is minimal with such a property.
If
\[
B(x) = \max\{|g(x)|,|g'(x)|,\dots,|g^{(r-1)}(x)|\} \hbox{~~and~~} C = \max \{1,|c_0|,\dots,|c_{r-1}|\},
\]
then, for any $x\in\bC$ and $k\geq 0$, we have
\[
|g^{(k)}(x)| \leq \left\{\begin{array}{ll} B(x) & \hbox{if~} k < r \\
(2\cdot C)^{k-r}\cdot r\cdot B(x) \cdot C \cdot & \hbox{if~} k\geq r. \end{array}\right.
\]
In particular, $|g^{(k)}(x)|\leq (2\cdot C)^{k-1}\cdot r\cdot B(x)\cdot C = 2^{k-1}\cdot r\cdot C^k\cdot B(x)$.
\end{lemma}
\begin{proof}
We know $g^{(r)} = \sum_{i=0}^{r-1} c_i g^{(i)}(x)$.  For any $k > r$, by differentiation,
we know
\[
g^{(k)} = \sum_{i=0}^{r-1} c_i g^{(i+k-r)}(x).
\]

We will now proceed by induction starting at $k = r$.  In particular,
\[
|g^{(r)}(x)| \leq \sum_{i=0}^{r-1} |c_i|\cdot |g^{(i)}(x)| \leq B(x)\cdot C\sum_{i=0}^{r-1} 1 = r \cdot B(x) \cdot C.
\]

For $k > r$ with $p = k - r$, we have
\[
\begin{array}{rcl}
|g^{(k)}(x)| &\leq& \displaystyle\sum_{i=0}^{r-1} |c_i|\cdot |g^{(i+p)}(x)| \leq
C \left(\displaystyle\sum_{i=0}^{\max \{r-1-p,0\} } |g^{(i+p)}(x)| + \displaystyle\sum_{i = \max\{0,r-p\}}^{r-1} |g^{(i+p)}(x)|\right) \\
&\leq& C \left(r\cdot B(x) + r\cdot B(x)\cdot C \displaystyle\sum_{i=r-p}^{r-1} (2\cdot C)^{i+p-r}\right) \\
&\leq& r\cdot B(x)\cdot C^2\left(1 + C^{p-1}\displaystyle\sum_{i=0}^{p-1} 2^i\right) \\
&\leq& 2^p\cdot r\cdot B(x)\cdot C^{p+1} \\
&=& (2\cdot C)^{k-r} \cdot r\cdot B(x) \cdot C.
\end{array}
\]
The remaining statement follows from the fact that $C\geq 1$ and $r\geq 1$.
\end{proof}

The following lemma will also be used to deduce our bound.

\begin{lemma}\label{Lemma:Sup}
If $\delta_0\geq0$ and $\alpha_1,\delta_1,\dots,\alpha_m,\delta_m\geq 1$, then
\[
\sup_{k\geq2} \left(\delta_0^{2(k-1)} + 2^{2(k-1)}\displaystyle\sum_{i=1}^m \left(\alpha_i^k \delta_i\right)^2\right)^{\frac{1}{2(k-1)}} \leq
\delta_0 + 2 \sum_{i=1}^m \alpha_i^2 \delta_i.
\]
\end{lemma}
\begin{proof}
Fix $k\geq2$.  Since $2(k-1)\geq2$ and $4(k-1)\geq2k$, we know $\alpha_i^{4(k-1)}\geq\alpha_i^{2k}$ and $\delta_i^{2(k-1)}\geq\delta_i^2$
for $i = 1,\dots,m$.  The lemma now follows since
\[
\begin{array}{rcl}
\left(\delta_0 + 2 \displaystyle\sum_{i=1}^m \alpha_i^2 \delta_i\right)^{2(k-1)}
&\geq&
    \delta_0^{2(k-1)} + 2^{2(k-1)}\left(\displaystyle\sum_{i=1}^m \alpha_i^2 \delta_i\right)^{2(k-1)} \\
&\geq&
    \delta_0^{2(k-1)} + 2^{2(k-1)} \displaystyle\sum_{i=1}^m \alpha_i^{4(k-1)} \delta_i^{2(k-1)} \\
&\geq&
    \delta_0^{2(k-1)} + 2^{2(k-1)} \displaystyle\sum_{i=1}^m \alpha_i^{2k} \delta_i^2.
\end{array}
\]
\end{proof}

Throughout the remainder of this section, we assume that
$\sF:\bC^{n+m}\rightarrow\bC^{n+m}$ is a polynomial-exponential system such that
there exists a polynomial system $P:\bC^{n+m}\rightarrow\bC^n$,
analytic functions $g_1,\dots,g_m:\bC\rightarrow\bC$,
and integers $\sigma_1,\dots,\sigma_m\in\{1,\dots,n\}$ such that
\begin{equation}\label{Eq:PolyExp}
\sF(x_1,\dots,x_n,y_1,\dots,y_m) = \left[\begin{array}{c} P(x_1,\dots,x_n,y_1,\dots,y_m)
\\ y_1 - g_1(x_{\sigma_1}) \\ \vdots \\ y_m - g_m(x_{\sigma_m}) \end{array}\right].
\end{equation}
Also, for $i = 1,\dots,n$, we define $d_i = \deg P_i$ and $D = \max d_i$.

We assume that each $g_i$ satisfies some nonzero linear homogeneous PDE with
complex coefficients.  For each $i = 1,\dots,m$, let $r_i$ be the smallest positive integer
such that there exists a nonzero linear function $\ell_i:\bC^{r_i+1}\rightarrow\bC$
with $\ell_i(g_i,g_i',\dots,g_i^{(r_i)}) = 0$.  By construction,
the coefficient of $z_{r_i}$ in $\ell_i(z_0,z_1,\dots,z_{r_i})$ must
be nonzero.  Upon rescaling $\ell_i$, we will assume that this coefficient is one, that is,
we have
\begin{equation}\label{Eq:LinearODE}
\ell_i(z_0,z_1,\dots,z_{r_i}) = z_{r_i} - c_{i,r_i-1}z_{r_i-1} - \cdots - c_{i,0}z_0
\end{equation}
which yields $g_i^{(r_i)} = \sum_{j=0}^{r_i-1} c_{i,j} g_i^{(j)}$.
We note that the minimal integer $r_i$ with such a property is called the {\em order} of $g_i$.

For example, for nonzero $\lambda,\mu\in\bC$, if $g_1(x) = e^{\lambda x}$, $g_2(x) = \cos(\mu x)$, and $g_3(x) = x \sin(x)$,
then the order of $g_i$ is $1, 2,$ and $4$, respectively.  The corresponding differential
equations are
\[
\frac{\d g_1}{\d x} - \lambda g_1 = 0, \
\frac{\d^2 g_2}{\d x^2} + \mu^2 g_2 = 0, \ \hbox{~and~} \
\frac{\d^4 g_3}{\d x^4} + 2 \frac{\d^2 g_3}{\d x^2} + g_3 = 0
\]
with linear functions
\[
\ell_1(z_0,z_1) = z_1 - \lambda z_0, ~~~ \ell_2(z_0,z_1,z_2) = z_2 + \mu^2 z_0, \ \hbox{~and~} \ \ell_3(z_0,z_1,z_2,z_3,z_4) = z_4 + 2 z_2 + z_0.
\]

The bound obtained in Proposition~\ref{Prop:GammaBound} depends upon $\mu(f,x)$
defined in (\ref{Eq:MuPoly}) for polynomial systems.  We extend this to
polynomial-exponential systems by defining
\begin{equation}\label{Eq:MuExp}
\mu(\sF,(x,y)) = \max\left\{1, \left\|D\sF(x,y)^{-1} \left[\begin{array}{cc} \Delta_{(d)}(x,y)\|P\| & \\ & I_m \end{array}\right]\right\|\right\}
\end{equation}
assuming that $D\sF(x,y)$ is invertible.  The matrix $\Delta_{(d)}(x,y)$ is the $n\times n$
diagonal matrix defined in (\ref{Eq:Delta}) and $I_m$ is the $m\times m$ identity matrix.
We note that (\ref{Eq:MuExp}) reduces to (\ref{Eq:MuPoly}) when $m = 0$.

The following theorem yields a bound for $\gamma(\sF,(x,y))$.

\begin{theorem}\label{Thm:ExpBound}
For $i = 1,\dots,m$ and $z\in\bC$, define
\[
B_i(z) = \max\{|g_i(z)|,\dots,|g_i^{(r_i-1)}(z)|\} \hbox{~~and~~} C_i = \max\{1,|c_{i,0}|,\dots,|c_{i,r_i-1}|\}.
\]
Then, for any $(x,y)\in\bC^{n+m}$ such that $D\sF(x,y)$ is invertible,
\begin{equation}\label{Eq:GammaBound}
\gamma(\sF,(x,y))\leq \mu(\sF,(x,y))\left(\frac{D^{3/2}}{2\|(x,y)\|_1} + 2 \sum_{i=1}^m C_i^2 \max\{1,r_i\cdot B_i(x_{\sigma_i})\}\right).
\end{equation}
\end{theorem}
\begin{proof}
Let $\sM = \left[\begin{array}{cc} \Delta_{(d)}(x,y)\|P\| & \\ & I_m \end{array}\right]$ and $k\geq2$.
We have
\[
\begin{array}{rcl}
\displaystyle\left\|\frac{D\sF(x,y)^{-1} D^k\sF(x,y)}{k!}\right\| &\leq&
\displaystyle\left\|D\sF(x,y)^{-1} \sM\right\| \left\|\frac{\sM^{-1} D^k\sF(x,y)}{k!}\right\| \\
&\leq&
\displaystyle\mu(\sF,(x,y)) \left\|\frac{\sM^{-1} D^k\sF(x,y)}{k!}\right\|.
\end{array}
\]
By Proposition~\ref{Prop:BoundK} and Lemma~\ref{Lemma:ExpBoundGen},
\[
\begin{array}{rcl}
\displaystyle\left\|\frac{\sM^{-1} D^k\sF(x,y)}{k!}\right\|^2 &=&
\displaystyle\sum_{i=1}^n \left\|\frac{D^kP_i(x,y)}{d_i^{1/2}\cdot\|(x,y)\|_1^{d_i-1}\cdot\|P\|\cdot k!}\right\|^2
+ \sum_{i=1}^m \left\|\frac{D^k g_i(x_{\sigma_i})}{k!}\right\|^2 \\
&\leq& \displaystyle\sum_{i=1}^n \frac{\|P_i\|^2}{\|P\|^2}\left(\frac{d_i^{3/2}}{2\|(x,y)\|_1}\right)^{2(k-1)} +
\sum_{i=1}^m \left(2^{k-1}\cdot r_i\cdot C_i^k\cdot B_i(x_{\sigma_i})\right)^2 \\
&\leq& \displaystyle\left(\frac{D^{3/2}}{2\|(x,y)\|_1}\right)^{2(k-1)} +
2^{2(k-1)} \sum_{i=1}^m \left(r_i\cdot C_i^k\cdot B_i(x_{\sigma_i})\right)^2.
\end{array}
\]
This yields
\[
\begin{array}{rcl}
\gamma(\sF,(x,y)) &=& \displaystyle\sup_{k\geq2} \left\|\frac{D\sF(x,y)^{-1} D^k\sF(x,y)}{k!}\right\|^{\frac{1}{k-1}} \\
&\leq& \displaystyle \mu(\sF,(x,y))\sup_{k\geq2}
\left(\left(\frac{D^{3/2}}{2\|(x,y)\|_1}\right)^{2(k-1)} +
2^{2(k-1)} \sum_{i=1}^m \left(r_i\cdot C_i^k\cdot B_i(x_{\sigma_i})\right)^2\right)^{\frac{1}{2(k-1)}} \\
&\leq& \displaystyle \mu(\sF,(x,y))\sup_{k\geq2}
\left(\left(\frac{D^{3/2}}{2\|(x,y)\|_1}\right)^{2(k-1)} + \right.\\
& & ~~~~~~~~~~~~~~~~~~~~~~~~~~~~~~~~ \left.
\displaystyle 2^{2(k-1)} \sum_{i=1}^m \left(C_i^k \max\{1,r_i\cdot B_i(x_{\sigma_i})\}\right)^2\right)^{\frac{1}{2(k-1)}}.
\end{array}
\]
The result now follows from Lemma~\ref{Lemma:Sup}.
\end{proof}

\begin{remark}
When $m = 0$, the bounds provided in Theorem~\ref{Thm:ExpBound} and Proposition~\ref{Prop:GammaBound} agree.
\end{remark}

The following is an algorithm to certify approximate solutions of $\sF = 0$.

\begin{description}
  \item[Procedure $B = \AlgCertX(\sF,z)$]
  \item[Input] A polynomial-exponential system $\sF:\bC^{n+m}\rightarrow\bC^{n+m}$ and a point $z\in\bC^{n+m}$.
  \item[Output] A boolean which is \mbox{\it True} if $z$ can be certified as an approximate solution of $\sF = 0$, otherwise, \mbox{\it False}.
  \item[Begin] \hskip -0.1in
  \begin{enumerate}
    \item If $\sF(z) = 0$, return \mbox{\it True}, otherwise, if $D\sF(z)$ is not invertible, return \mbox{\it False}.
    \item Set $\beta := \|D\sF(z)^{-1}\sF(z)\|$ and $\gamma$ to be the upper bound for $\gamma(\sF,z)$ provided in
          Theorem~\ref{Thm:ExpBound}.
    \item If $\beta\cdot\gamma < \displaystyle\frac{13 - 3 \sqrt{17}}{4}$, return \mbox{\it True}, otherwise return \mbox{\it False}.
  \end{enumerate}
\end{description}

The algorithms {\bf CertifyDistinctSoln} and {\bf CertifyRealSoln} from \cite{alphaCertifiedPaper}
apply to polynomial-exponential systems using the bound provided in Theorem~\ref{Thm:ExpBound}.
The algorithm {\bf CertifyDistinctSoln} determines if two approximate solutions have distinct
associated solutions.  The algorithm {\bf CertifyRealSoln} applies to
polynomial-exponential systems $\sF$ such that $N_\sF(\bR^{n+m})\subset\bR^{n+m}$
and determines if the associated solution to a given approximate solution is real.

We conclude this section with a refinement of Theorem~\ref{Thm:ExpBound} applied
to polynomial-exponential systems depending on $\exp$, $\sin$, $\cos$,
$\sinh$, and $\cosh$.  This refinement uses the following lemma.

\begin{lemma}\label{Lemma:Sup2}
If $\lambda_0,\dots,\lambda_m\geq0$ and $\mu_1,\dots,\mu_m\geq2$, then
\[
\sup_{k\geq2} \left(\lambda_0^{2(k-1)} + \displaystyle\sum_{i=1}^m \left(\frac{\mu_i \lambda_i^{k-1}}{k!}\right)^2\right)^{\frac{1}{2(k-1)}} \leq
\lambda_0 + \frac{1}{2}\sum_{i=1}^m \mu_i \lambda_i.
\]
\end{lemma}
\begin{proof}
Fix $k\geq2$.  Since $2(k-1)\geq2$ and $\mu_i\geq2$, we know $\displaystyle\left(\frac{\mu_i}{2}\right)^{2(k-1)}\geq\left(\frac{\mu_i}{2}\right)^2$.
The lemma follows from
\[
\begin{array}{rcl}
\left(\lambda_0 + \displaystyle\frac{1}{2}\sum_{i=1}^m \mu_i \lambda_i\right)^{2(k-1)}
&\geq&
    \lambda_0^{2(k-1)} + \left(\displaystyle\sum_{i=1}^m \frac{\mu_i \lambda_i}{2}\right)^{2(k-1)} \\
&\geq&
    \lambda_0^{2(k-1)} + \displaystyle\sum_{i=1}^m \left(\frac{\mu_i}{2}\right)^{2(k-1)} \lambda_i^{2(k-1)} \\
&\geq&
    \lambda_0^{2(k-1)} + \displaystyle\sum_{i=1}^m \frac{\mu_i^2 \lambda_i^{2(k-1)}}{2^2} \\
&\geq&
    \lambda_0^{2(k-1)} + \displaystyle\sum_{i=1}^m \left(\frac{\mu_i \lambda_i^{(k-1)}}{k!}\right)^2.
\end{array}
\]
\end{proof}

Let $a,b,c,e,h\in\bZ_{\geq0}$,
$\delta_i,\epsilon_j,\zeta_k,\eta_p,\kappa_q\in\bC$,
and $\sigma_i,\tau_j,\phi_k,\chi_p,\psi_q\in\{1,\dots,n\}$.
The following considers the following polynomial-exponential system
\begin{align}
\sG(x_1,\dots,x_n,u_1,\dots,u_a,v_1,\dots,v_b,w_1,\dots,w_c,y_1,\dots,y_d,z_1,\dots,z_e) = ~~~~~~~~~\nonumber \\
~~~~\left[\begin{array}{c}
P(x_1,\dots,x_n,u_1,\dots,u_a,v_1,\dots,v_b,w_1,\dots,w_c,y_1,\dots,y_d,z_1,\dots,z_e) \\
\begin{array}{cc}
u_i - \exp(\delta_i x_{\sigma_i}), & i = 1,\dots,a \\
v_j - \sin(\epsilon_j x_{\tau_j}), & j = 1,\dots,b \\
w_k - \cos(\zeta_k x_{\phi_k}), & k = 1,\dots,c  \\
y_p - \sinh(\eta_p x_{\chi_p}), & p = 1,\dots,e  \\
z_q - \cosh(\kappa_q x_{\psi_q}), & q = 1,\dots,h \end{array}
\end{array}\right].\label{Eq:ReductionG}
\end{align}

\begin{corollary}\label{Cor:ExpBound}
Let $\sG$ be defined as in (\ref{Eq:ReductionG}) where $P:\bC^N\rightarrow\bC^n$ is
a polynomial system with $N = n+a+b+c+e+h$, $d_i = \deg P_i$ and $D = \max d_i$.
For any $\lambda$, $\theta\in\bC$, define
\[ \begin{array}{l}
A(\lambda,\theta) = \max\{|\lambda|,|\lambda^2 \exp(\lambda \theta)/2|\}, \\
B(\lambda,\theta) = \max\{|\lambda|,|\lambda^2 \sin(\lambda \theta)/2|,|\lambda^2 \cos(\lambda \theta)/2|\}, \hbox{~~and~~} \\
C(\lambda,\theta) = \max\{|\lambda|,|\lambda^2 \sinh(\lambda \theta)/2|,|\lambda^2 \cosh(\lambda \theta)/2|\}.
\end{array}
\]
Then, for any $X=(x,u,v,w,y,z)\in\bC^N$ such that $D\sG(X)$ is invertible,
\begin{align}
\gamma(\sG,X)\leq \mu(\sG,X)\left(\frac{D^{3/2}}{2\|X\|_1} + \sum_{i=1}^a A(\delta_i,x_{\sigma_i}) \right. &
    + \sum_{j=1}^b B(\epsilon_j,x_{\tau_j}) + \sum_{k=1}^c B(\zeta_k,x_{\phi_k}) \nonumber \\
 & ~~~~~~~  \left. + \sum_{p=1}^e C(\eta_p,x_{\chi_p}) + \sum_{q=1}^h C(\kappa_q,x_{\psi_q})\right). \label{Eq:GammaBoundG}
\end{align}
\end{corollary}
\begin{proof}
Let $k\geq 2$.  The following table lists the bounds on the higher derivatives together
with associated quantities $\lambda$ and $\mu$ used when applying Lemma~\ref{Lemma:Sup2}.
\[
\begin{array}{c|c|c|c}
g(x) & \hbox{bound for~}|g^{(k)}(x)| & \lambda & \mu \\
\hline \hline
\exp(\theta x) & |\theta^k \exp(\theta x)| & |\theta| & \max\{2,|\theta \exp(\theta x)|\} \\
\hline
\sin(\theta x) & \multirow{2}{*}{$|\theta^k|\max\{|\sin(\theta x)|,|\cos(\theta x)|\}$}
& \multirow{2}{*}{$|\theta|$} & \multirow{2}{*}{$\max\{2,|\theta \sin(\theta x)|,|\theta \cos(\theta x)\}$} \\
\cos(\theta x) & & & \\
\hline
\sinh(\theta x) & \multirow{2}{*}{$|\theta^k|\max\{|\sinh(\theta x)|,|\cosh(\theta x)|\}$}
& \multirow{2}{*}{$|\theta|$} & \multirow{2}{*}{$\max\{2,|\theta \sinh(\theta x)|,|\theta \cosh(\theta x)\}$} \\
\cosh(\theta x) & & & \\
\end{array}
\]
The result now follows immediately by modifying the proof of Theorem~\ref{Thm:ExpBound}
incorporating the bounds presented in this table together with Lemma~\ref{Lemma:Sup2}.
Based on Lemma~\ref{Lemma:Sup2}, the functions $A$, $B$, and $C$ are one-half of the
product of the entries in the $\lambda$ and $\mu$ columns.
\end{proof}

\section{Approximating solutions}\label{Sec:ApproxSolns}

In order to certify that a point is an approximate solution of $\sF = 0$,
where $\sF$ is a polynomial-exponential system, one needs to first have a candidate
point.  In some applications, candidate points arise naturally from the
formulation of the problem.  One systematic approach to yield candidate
points is to replace each analytic function $g_i$ by a polynomial $g_i^p$
and solve the resulting polynomial system, namely
\begin{equation}\label{Eq:PolyTruncate}
\sF^p(x_1,\dots,x_n,y_1,\dots,y_m) = \left[\begin{array}{c}
P(x_1,\dots,x_n,y_1,\dots,y_m) \\
y_1 - g_1^p(x_{\sigma_1}) \\ \vdots \\ y_m - g_m^p(x_{\sigma_m}) \end{array}\right].
\end{equation}
When the degree of the polynomial approximations are sufficiently large,
the numerical solutions of $\sF^p = 0$ are candidates for being approximate
solutions of $\sF = 0$.  In Section~\ref{Sec:Regen}, we discuss using
regeneration \cite{Regen} to solve $\sF^p = 0$.

If a numerical solution of $\sF^p = 0$ is not an approximate solution of $\sF = 0$,
one can try to apply Newton's method for $\sF$ directly to these points
to possibly yield an approximate solution of $\sF = 0$.  Another
approach is to construct a homotopy between $\sF^p$ and $\sF$, and numerically
approximate the endpoint of the path starting with a solution of $\sF^p = 0$.
We note that neither method is guaranteed to yield an approximate solution of $\sF = 0$.

\subsection{Regeneration and polynomial-exponential systems}\label{Sec:Regen}

Regeneration \cite{Regen} solves a polynomial system by using solutions to
related, but easier to solve, polynomial systems.  In particular, we will 
utilize the linear product \cite{LinearProduct} structure of $\sF^p$ 
in (\ref{Eq:PolyTruncate}).  

Suppose that $g$ is a univariate polynomial of degree $d$.  The
polynomial $y - g(x)$ has a linear product structure of 
\[
\langle x,y,1 \rangle \times \underbrace{\langle x,1 \rangle \times \cdots \times
\langle x,1 \rangle}_{d-1 \hbox{\footnotesize ~times}}.
\]
That is, $y - g(x)$ is a finite sum of polynomials
of the~form~$L_1(x,y) \cdots L_d(x,y)$~where 
\[
L_1(x,y) = a y + b_1 x + c_1 \hbox{~~and, for $i = 2,\dots,d$,~~} L_i(x,y) = b_i x + c_i
\]
for some $a,b_i,c_i\in\bC$.  

For $i = 1,\dots,m$, let $r_i = \deg g_i^p$ and $a_i,b_{i,1},\dots,b_{i,r_i}\in\bC$.
Similar to the algorithms proposed in \cite{Regen}, we note that the following arguments 
and proposed algorithm depend on the genericity of $a_i$ and $b_{i,j}$.
Define
\[
  L_{i,1}(x,y) = a_i y + b_{i,1} x + 1 \hbox{~~and, for $j = 2,\dots,r_i$~~} L_{i,j}(x,y) = b_{i,j} x + 1.
\]
Let $\nu = (\nu_1,\dots,\nu_m)$ such that $1 \leq \nu_i \leq r_i$.
Consider the polynomial systems $\sQ_\nu:\bC^{n+m}\rightarrow\bC^{n+m}$
defined by
\begin{equation}\label{Eq:Qext}
\sQ_\nu(x_1,\dots,x_n,y_1,\dots,y_m) = \left[\begin{array}{c} P(x_1,\dots,x_n,y_1,\dots,y_m)
 \\ L_{1,\nu_1}(x_{\sigma_1},y_1) \\ \vdots \\ L_{m,\nu_m}(x_{\sigma_m},y_m) \end{array}\right].
\end{equation}
For $\bone = (1,\dots,1)$, we first compute the solutions of $\sQ_{\bone} = 0$.
We note that in practice, $\sQ_{\bone}$ is solved by working intrinsically on the linear
space defined by $L_{1,\nu_1}(x_{\sigma_1},y_1) = \cdots = L_{m,\nu_m}(x_{\sigma_m},y_m) = 0$.
Numerical approximations of these solutions can be obtained using standard 
numerical solving methods for square polynomial systems (see \cite{SW05,AlgKin})
including, for example, polyhedral homotopies \cite{Polyhedral} 
or basic regeneration \cite{Regen}.

In order to compute the nonsingular isolated solutions of $\sF^p = 0$,
we need to compute the nonsingular isolated solutions of 
$\sQ_\nu = 0$ for all possible $\nu$.
By the theory of coefficient-parameter homotopies \cite{CoeffParam},
the nonsingular isolated solutions of $\sQ_\nu = 0$ can be obtained
by using a homotopy from $\sQ_{\bone}$ to $\sQ_\nu$ starting with
the nonsingular isolated solutions of $\sQ_{\bone} = 0$.  
We note that if $i\neq j$ such that $\sigma_i = \sigma_j$
and $\nu_i,\nu_j > 1$, then $\sQ_\nu = 0$ has no solutions.  

After solving $\sQ_\nu = 0$ for all possible $\nu$, we thus have 
all nonsingular isolated solutions of
\begin{equation}\label{Eq:sP}
\sP(x_1,\dots,x_n,y_1,\dots,y_m) =
\left[ \begin{array}{c} P(x_1,\dots,x_n,y_1,\dots,y_m) \\
\prod_{j=1}^{r_1} L_{1,j}(x_{\sigma_1},y_1) \\ \vdots \\
\prod_{j=1}^{r_m} L_{m,j}(x_{\sigma_m},y_m) \end{array} \right] = 0.
\end{equation}
The final step is to use a homotopy deforming $\sP$ to $\sF^p$
starting with the nonsingular isolated solutions of $\sP = 0$.
The finite endpoints of this homotopy form a superset of the isolated
nonsingular solutions of $\sF^p = 0$.

\section{Implementation details and examples}\label{Sec:Examples}

The certification of polynomial-exponential systems is implemented
in \alphaCertS \cite{alphaCertified}.  The systems must be of the form
$\sG$ in (\ref{Eq:ReductionG}) where the coefficients of $P$
as well as the constant in the argument of $\exp$, $\sin$, $\cos$, $\sinh$, and
$\cosh$ must be rational complex numbers.  
The bound for~$\gamma$ presented in (\ref{Eq:GammaBoundG})
is computed.  Due to the nature of exponential
functions, the computations are performed using arbitrary precision floating
point arithmetic.  Since floating point errors arising from the internal computations
are not fully controlled, the results of \alphaCertS for polynomial-exponential
systems are said to be {\em soft certified}.  See Appendix~\ref{Sec:Appendix}
and \cite{alphaCertifiedPaper,alphaCertified}
for more details regarding input syntax, internal computations, and output.

In the following examples, we used \BertiniS \cite{Bertini} 
and \alphaCertS on a 2.4 GHz Opteron 250 processor 
running 64-bit Linux with $8$ GB of memory.
All files for running these examples can be found at 
\url{www.nd.edu/~jhauenst/PolyExp}.

\subsection{A rigid mechanism}\label{Sec:RRdyad}

Consider the algebraic kinematics problem \cite{AlgKin}
of the inverse kinematics of the RR dyad.  The RR dyad,
which is displayed in Figure~\ref{Fig:RR}, consists of two
legs of fixed length, say $a_1$ and $a_2$, which are connected
by a pin joint.  The mechanism is anchored with a pin joint at
the point $O$, which we take as the origin.  Given a point
$E = (e_1,e_2)$, the problem is compute the angles $\theta_1$
and $\theta_2$ so that the end of the second leg is at $E$.
That is, we want to solve $f(\theta_1,\theta_2) = 0$ where
\[
f(\theta_1,\theta_2) = \left[\begin{array}{c} a_1 \cos(\theta_1) + a_2 \cos(\theta_2) - e_1 \\
a_1 \sin(\theta_1) + a_2 \sin(\theta_2) - e_2 \end{array}\right].
\]
The polynomial-exponential system $\sG:\bC^6\rightarrow\bC^6$ of the form (\ref{Eq:ReductionG}) is
\begin{equation}\label{Eq:SystemG}
\sG(\theta_1,\theta_2,y_1,y_2,y_3,y_4) = \left[\begin{array}{c}
a_1 y_3 + a_2 y_4 - e_1 \\
a_1 y_1 + a_2 y_2 - e_2 \\
y_1 - \sin(\theta_1) \\
y_2 - \sin(\theta_2) \\
y_3 - \cos(\theta_1) \\
y_4 - \cos(\theta_2) \end{array}\right].
\end{equation}

\begin{figure}
  \begin{center}
     \begin{picture}(30,30)(0,0)
       \put(-20,-20){\includegraphics{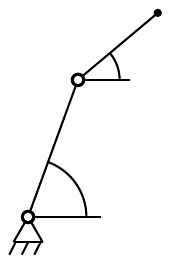}}
       \put(-9,3){$O$}
       \put(5,38){$a_1$}
       \put(29,27){$\theta_1$}
       \put(52,78){$\theta_2$}
       \put(41,92){$a_2$}
       \put(71,100){$E$}
     \end{picture}
  \end{center}
  \caption{RR dyad}\label{Fig:RR}
\end{figure}

Since $\theta_i$ only appears in $f$ as arguments of the sine and cosine functions,
we can compute solutions of $f = 0$ by using the solutions of
a related polynomial system.  In particular, consider the polynomial system
$g:\bC^4\rightarrow\bC^4$ obtained by replacing $\sin(\theta_i)$ and $\cos(\theta_i)$
with $s_i$ and $c_i$, respectively, and adding the Pythagorean identities, namely
\[
g(s_1,s_2,c_1,c_2) = \left[\begin{array}{c} a_1 c_1 + a_2 c_2 - e_1 \\
a_1 s_1 + a_2 s_2 - e_2 \\ s_1^2 + c_1^2 - 1 \\ s_2^2 + c_2^2 - 1 \end{array}\right].
\]
Given a solution of $g = 0$, solutions of $f = 0$ are generated using
either the $\arcsin$ or $\arccos$ functions.  Moreover,
it is easy to verify that, for general $a_i,e_i\in\bC$, $g = 0$ has
two solutions and thus $f = 0$ has two $2\pi$-periodic families of solutions.

Consider the inverse kinematics problem with $a_1 = 3$, $a_2 = 2$,
and $E = (1, 3.5)$.  We used \BertiniS to numerically approximate
the two solutions of $g = 0$.  For demonstration, consider the two digit rational
approximations of the solutions
\[
X_1 = \frac{1}{100}(65,77,76,-64) \hbox{~~and~~} X_2 =\frac{1}{100}(95,32,-30,95).
\]
The certified upper bounds for $\alpha(g,X_i)$ computed by \alphaCertS
using exact rational arithmetic and rounded to four digits
are $0.0736$ and $0.0788$, respectively.  Hence,
$X_1$ and $X_2$ are both approximate solutions of $g = 0$.
Furthermore, \alphaCertS certified that the associated solutions
are distinct and real.

We now consider two corresponding approximations to solutions of $\sG = 0$ namely
\begin{equation}\label{Eq:Z}
Z_1 = (0.711,2.261,0.65,0.77,0.76,-0.64) \hbox{~and~} Z_2 = (1.874,0.324,0.95,0.32,-0.30,0.95).
\end{equation}
The upper bounds for $\alpha(\sG,Z_i)$ computed by \alphaCertS
using 96-bit floating point arithmetic and rounded to four digits
are $0.1265$ and $0.1355$, respectively.  In order to reduce the effect of
roundoff errors, we also used 1024-bit floating point arithmetic
and obtained the same four digit value.  Hence, \alphaCertS has
soft certified that $Y_1$ and $Y_2$ are both approximate
solutions of $\sG = 0$.
Furthermore, \alphaCertS has soft certified that the associated solutions
are distinct and real.
Table~\ref{Tab:RRdyad} lists the Newton residuals
computed by \alphaCertS using 4096-bit precision
which demonstrates the quadratic convergence of Newton's method.

\begin{table}
\centering
\begin{tabular}{|c|c|c|}
\hline
$k$ & $\beta(\sG,N_\sG^k(Z_1))$ & $\beta(\sG,N_\sG^k(Z_2))$ \\
\hline
0 & $4.94\cdot 10^{-3}$ & $5.26\cdot 10^{-3}$ \\
\hline
1 & $7.46\cdot 10^{-9}$ & $6.29\cdot 10^{-9}$ \\
\hline
2 & $1.21\cdot 10^{-17}$ & $8.86\cdot 10^{-18}$ \\
\hline
3 & $3.65\cdot 10^{-35}$ & $2.01\cdot 10^{-35}$ \\
\hline
4 & $3.56\cdot 10^{-70}$ & $1.10\cdot 10^{-70}$ \\
\hline
5 & $3.56\cdot 10^{-140}$ & $3.41\cdot 10^{-141}$ \\
\hline
6 & $3.50\cdot 10^{-280}$ & $3.21\cdot 10^{-282}$ \\
\hline
7 & $3.44\cdot 10^{-560}$ & $2.90\cdot 10^{-564}$ \\
\hline
\end{tabular}
\caption{Newton residuals for $\sG$}\label{Tab:RRdyad}
\end{table}

By using Euler's formula, we could alternatively use the 
polynomial-exponential system
$\sG':\bC^6\rightarrow\bC^6$ of the form (\ref{Eq:ReductionG}) where
\[
\sG'(\theta_1,\theta_2,x_1,x_2,y_1,y_2) = \left[\begin{array}{c}
a_1 x_1 + a_2 x_2 - e_1 + ie_2 \\
a_1 y_1 + a_2 y_2 - e_1 - ie_2\\
x_1 y_1 - 1 \\
x_2 y_2 - 1 \\
y_1 - \exp(i\theta_1) \\
y_2 - \exp(i\theta_2) \end{array}\right]
\]
and $i = \sqrt{-1}$.  Consider the two points
\[
\begin{array}{l}
W_1 = (0.711,2.261,0.758-0.653i,-0.637-0.771i,0.758+0.653i,-0.637+0.771i) \hbox{~~and} \\
W_2 = (1.874,0.324,-0.299-0.954i,0.948-0.318i,-0.299+0.954i,0.948+0.318i).
\end{array}
\]
The upper bounds for $\alpha(\sG',W_i)$ computed by \alphaCertS
using both 96-bit and 1024-bit floating point arithmetic and 
rounded to four digits are $0.1492$ and $0.1422$, respectively.  
In particular, \alphaCertS soft certified that $W_1$ and $W_2$ are 
both approximate solutions of $\sG' = 0$ with distinct associated solutions.  

Finally, consider the polynomial system obtained by replacing the sine and cosine functions in $f$
with a third and second degree truncated Taylor series approximation, respectively, centered at the origin, namely  
\[
f^p(\theta_1,\theta_2) = \left[\begin{array}{c} a_1 (1 + \theta_1^2/2) + a_2 (1 + \theta_2^2/2) - e_1 \\
a_1 (\theta_1 + \theta_1^3/6) + a_2 (\theta_2 + \theta_2^3/6) - e_2 \end{array}\right].
\]
The system of equations $f^p = 0$ has six solutions and yield six solutions of $f = 0$
upon deforming $f^p$ to $f$.  These six solutions split into two groups of three based on
the values of $\sin(\theta_i)$ and $\cos(\theta_i)$ corresponding to 
the two families of solutions of $f = 0$.

\subsection{A compliant mechanism}\label{Sec:4bar}

In \cite{CompliantMech}, Su and McCarthy study a 
polynomial-exponential system modeling a compliant 
four-bar linkage displayed in \cite[Fig.~4]{CompliantMech}.  
Upon solving a related polynomial system and applying 
Newton's method, they conclude based on the numerical results 
that a specific compliant four-bar linkage 
has two stable configurations.  We will first 
use \alphaCertS to certify that their numerical approximations 
of the two stable configurations are indeed approximate solutions.
Afterwards, we will use the approaches of 
Section~\ref{Sec:ApproxSolns} to recompute these two 
stable configurations.

The polynomial-exponential system $f:\bC^5\rightarrow\bC^5$ modeling a
compliant four-bar linkage is 
\[
f(\alpha, \theta_1, \theta_2, \nu_1, \nu_2) = 
\left[\begin{array}{c}
R(\alpha)(W_2 - W_1) + G_1 + r_1 cs(\theta_1) - G_2 - r_2 cs(\theta_2) \\
R(\alpha)(W_2 - W_1)\nu_1 + r_1 cs(\theta_1) - r_2 cs(\theta_2)\nu_2  \\
k_1 (\alpha - \alpha^0 - \theta_1 + \theta_1^0) (\nu_1 - 1) 
    + k_2 (\alpha - \alpha^0 - \theta_2 + \theta_2^0) (\nu_1 - \nu_2)
\end{array}\right]
\]
where 
\[
R(\alpha) = \left[\begin{array}{cc} \cos(\alpha) & -\sin(\alpha) \\ \sin(\alpha) & \cos(\alpha) \end{array}\right]
\hbox{~~and~~}
cs(\theta) = \left[\begin{array}{c} \cos(\theta) \\ \sin(\theta) \end{array}\right].
\]
We note that each of the first two lines in $f$ consists of two functions.
Additionally, $f$ is not algebraic since $X$, $\sin(X)$, and $\cos(X)$ all appear in $f$
when $X$ is either $\alpha$, $\theta_1$, or $\theta_2$.  

The values for the specific linkage under consider are
\[
W_1 = \left[\begin{array}{c} -112.632 \\ -45.053 \end{array}\right],
W_2 = \left[\begin{array}{c} 112.632 \\ -45.053 \end{array}\right],
G_1 = \left[\begin{array}{c} 0 \\ 0 \end{array}\right],
G_2 = \left[\begin{array}{c} 100 \\ 0 \end{array}\right],
r_1 = r_2 = 250,
\]
\[
k_1 = 29250, k_2 = 5824.29, \theta_1^0 = 1.4486, \theta_2^0 = 0.925, \hbox{~and~} \alpha^0 = -0.2169.
\]
with numerical approximations for the stable configurations
\[
\begin{array}{l}
 A_1 = (-0.216933, 1.448567, 0.924966, 0.610174, 1.094669) \hbox{~~and~~} \\
 A_2 = (-1.516473, 0.131930, -0.875993, 1.570656, 1.668379). \end{array}
\]

The polynomial-exponential system $\sG:\bC^{11}\rightarrow\bC^{11}$ 
of the form (\ref{Eq:ReductionG}) is
\[
{\small
\sG(\alpha, \theta_1, \theta_2, \nu_1, \nu_2, y_1, \dots, y_6) =
\left[\begin{array}{c}
R(y_1,y_2)(W_2 - W_1) + G_1 + r_1 cs(y_3,y_4) - G_2 - r_2 cs(y_5,y_6) \\
R(y_1,y_2)(W_2 - W_1)\nu_1 + r_1 cs(y_3,y_4) - r_2 cs(y_5,y_6)\nu_2  \\
k_1 (\alpha - \alpha^0 - \theta_1 + \theta_1^0) (\nu_1 - 1) 
    + k_2 (\alpha - \alpha^0 - \theta_2 + \theta_2^0) (\nu_1 - \nu_2) \\
y_1 - \sin(\alpha) \\
y_2 - \cos(\alpha) \\
y_3 - \sin(\theta_1) \\
y_4 - \cos(\theta_1) \\
y_5 - \sin(\theta_2) \\
y_6 - \cos(\theta_2) \end{array}\right]
}
\]
where 
\[
R(y_1,y_2) = \left[\begin{array}{cc} y_2 & -y_1 \\ y_1 & y_2 \end{array}\right]
\hbox{~~and~~}
cs(w,z) = \left[\begin{array}{c} z \\ w \end{array}\right].
\]
Let $B_i = (A_i,Y_i)$ where 
\[
\begin{array}{l}
 Y_1 = (-0.215236, 0.976562, 0.992539, 0.121925, 0.798600, 0.601862) \hbox{~~and~~} \\
 Y_2 = (-0.998525, 0.0542970, 0.131547, 0.991310, -0.768180, 0.640235). \end{array}
\]
The upper bounds for $\alpha(\sG,B_i)$ computed by \alphaCertS
using both 96-bit and 1024-bit floating point arithmetic and rounded to four digits
are $0.0166$ and $0.0427$, respectively.  In particular, \alphaCertS has soft 
certified that $B_1$ and $B_2$ are both approximate
solutions of $\sG = 0$.  Furthermore, \alphaCertS has soft certified 
that the associated solutions are distinct and real.

The formulation of the polynomial-exponential system can have an adverse effect
on certifying solutions.  For example, consider the polynomial-exponential system
$\sG':\bC^{11}\rightarrow\bC^{11}$ obtained by replacing the 
$7^{th}$, $9^{th}$, and $11^{th}$ functions of $\sG$ with
\[
y_1^2 + y_2^2 - 1, \ y_3^2 + y_4^2 - 1, \hbox{~and~}\ y_5^2 + y_6^2 - 1.
\]
Clearly, every solution of $\sG = 0$ must also be a solution of $\sG' = 0$.  
Table~\ref{Tab:Compliant} compares the bounds for $\alpha$ and $\gamma$
and the value of $\beta$ for $\sG$ and $\sG'$ at $B_1$ and $B_2$ computed by \alphaCert.
This table shows that the bounds computed for $\alpha(\sG',B_i)$ and $\gamma(\sG',B_i)$ 
are three orders of magnitude larger than the bounds computed for $\alpha(\sG,B_i)$ and $\gamma(\sG,B_i)$.
In particular, due to the larger bounds, \alphaCertS is unable to certify that 
$B_1$ and $B_2$ are approximate solutions of $\sG' = 0$.  If we replace $B_i$ with $N_{\sG'}(B_i)$,
then \alphaCertS is able to soft certify that the resulting points are approximate solutions of $\sG' = 0$
using both 96-bit and 1024-bit precision.

\begin{table}
\centering
\begin{tabular}{|c|c|c|c|c|c|c|}
\hline
  & \multicolumn{2}{c|}{bound for} & \multicolumn{2}{c|}{approximation of} & \multicolumn{2}{c|}{bound for} \\
\cline{2-7}
$F$ & $\alpha(F,B_1)$ & $\alpha(F,B_2)$ & $\beta(F,B_1)$ & $\beta(F,B_2)$ & $\gamma(F,B_1)$, & $\gamma(F,B_2)$ \\
\hline
$\sG$ & $1.66\cdot 10^{-2}$ & $4.27\cdot 10^{-2}$ & $8.08\cdot10^{-7}$ 
                      & $1.06\cdot10^{-6}$ & $2.05\cdot10^4$ & $4.02\cdot10^4$ \\
\hline
$\sG'$ & $11.9$ & $42.5$ & $8.08\cdot10^{-7}$ & $1.06\cdot10^{-6}$ & $1.47\cdot10^7$ & $4.00\cdot10^7$ \\
\hline
\end{tabular}
\caption{Values obtained for $\sG$ and $\sG'$ at $B_1$ and $B_2$}\label{Tab:Compliant}
\end{table}

We now consider solving a polynomial system obtained
by replacing the sine and cosine functions with a fifth and fourth degree truncated 
Taylor series approximation, respectively, centered at the origin.  
Let the polynomial system $P:\bC^{11}\rightarrow\bC^5$ consists of the first 
five functions in $\sG$.  In particular, $P$ consists of 
two linear and three quadratic polynomials
and thus has total degree of the polynomial $\sQ_\nu$ 
defined in (\ref{Eq:Qext}) has total degree $2^3 = 8$.  

Since we are using fifth and fourth degree polynomial approximations
for the sine and cosine functions, respectively, we have $r_i = 5$ if $i$ is odd and $r_i = 4$ if $i$ is even.
We picked random $a_i,b_{i,j}\in\bC$ for $i = 1,\dots,6$ and $j = 1,\dots,r_i$
and used \BertiniS to solve each $\sQ_{\nu} = 0$.  
In total, this produced numerical approximations to $356$ nonsingular 
isolated solutions of $\sP = 0$ where~$\sP$ is defined in (\ref{Eq:sP}).

The tracking of the $356$ paths from $\sP$ to the polynomial approximation, $\sG^p$, 
of $\sG$ produced $120$ points which became the start points for the homotopy 
deforming $\sG^p$ to $\sG$.  This homotopy yielded $93$ numerical approximations to 
solutions of $\sG = 0$.  By using both 96-bit and 1024-bit floating point arithmetic, 
\alphaCertS soft certified that each of these $93$ 
points are indeed approximate solutions with distinct associated solutions
Moreover, this computation soft certified that $65$ have real associated solutions,
two of which are the two stable configurations computed in \cite{CompliantMech}.

\section{Conclusion}\label{Sec:Conclusion}

One key to certification using $\alpha$-theory
is the ability to compute a bound on $\gamma$, which is defined
in terms of all higher order derivatives.  For polynomial systems,
where there are only finitely many nonzero derivatives,
Shub and Smale developed the bound presented in Proposition~\ref{Prop:GammaBound}.
This bound is based on first order derivatives, coefficients and degrees of the polynomials, 
and the point of interest.  
Theorem~\ref{Thm:ExpBound} extends this bound to polynomial-exponential systems
and is implemented in {\tt alphaCertified}.

The computationally expensive part of computing the bound 
on $\gamma$ is the linear algebra computations required to compute~$\mu$ 
as defined in \eqref{Eq:MuExp}.  Thus, the restriction on the size of the 
systems for which the bound could be computed arises from the 
linear algebra algorithms used.  
Even though large systems could be investigated, the fact that
this produces an upper bound of~$\gamma$ means that $\beta$ will need
to be smaller in order to certify an approximate solution.  
Therefore, the use of this bound may induce additional computational 
cost via higher precision.  

Since the certification approach presented for polynomial-exponential systems
is based on the quadratic convergence of Newton's method and $\alpha$-theory,
we limit our focus to certifying nonsingular solutions to square systems.
Even though Newton's method near singular solutions can have a variety of behavior, 
e.g., see \cite{GO}, one can attempt the certification method~at~any~point.

\pdfbookmark[0]{References}{thebibliography}

\appendix

\pdfbookmark[0]{Appendix}{Sec:Appendix}

\section{Using {\tt alphaCertified}}\label{Sec:Appendix}

As a demonstration of using {\tt alphaCertified}, we consider the polynomial-exponential
system~$\sG$ in~\eqref{Eq:SystemG} where
$a_1 = 3$, $a_2 = 2$, and $E = (1,3.5)$ along with the points $Z_i$ in \eqref{Eq:Z}.

\subsection{Input}

We describe the three required files: input system, points, and configurations.  

\subsubsection*{Input system}

In order to describe the system $\sG$, which is of the required
form \eqref{Eq:ReductionG}, we first list the 
total number of variables, $6$, and the number of polynomials, $2$.
With this setup, {\tt alphaCertified} assumes that 
the last four variables will be defined in terms of the first two variables, 
which are described after the polynomials.  
Since the system is assumed to be exact, the real and imaginary parts of
all numbers listed in this file must be rational.

Each polynomial is represented as a sum of monomials.  Thus, 
we list the total number of monomials in the polynomial (both have
$3$ terms) followed by a description of each monomial.  A monomial is described
by the entries of the exponent vector followed by the real and imaginary parts of the coefficient.

The relations for the last four polynomials are described by the variable number
for which the analytic function depends upon, a string indicating which
analytic function (``X'' for $\exp()$, ``S'' for $\sin()$, ``C'' for $\cos()$, ``SH'' for $\sinh()$,
and ``CH'' for $\cosh()$), and the real and imaginary parts of the corresponding constant.

Figure~\ref{Fig:InputSystem} lists the contents of this file, which we name {\tt inputSystem}, 
along with comments.

\begin{figure}[!h]
\centering
$$\begin{array}{ll}
\mbox{\tt 6 2} & \mbox{number of variables and number of polynomials} \\ \\
\mbox{\tt 3} & \mbox{number of terms in first polynomial: $3y_3 + 2y_4 - 1$} \\
\mbox{\tt 0 0 0 0 1 0 3 0} & \mbox{$3y_3$} \\
\mbox{\tt 0 0 0 0 0 1 2 0} & \mbox{$2y_4$} \\
\mbox{\tt 0 0 0 0 0 0 -1 0} & \mbox{$-1$} \\
\\ 
\mbox{\tt 3} & \mbox{number of terms in second polynomial: $3y_1 + 2y_2 - 7/2$} \\
\mbox{\tt 0 0 1 0 0 0 3 0} & \mbox{$3y_1$} \\
\mbox{\tt 0 0 0 1 0 0 2 0} & \mbox{$2y_2$} \\
\mbox{\tt 0 0 0 0 0 0 -7/2 0} & \mbox{$-7/2$} \\
\\
\mbox{\tt 1 S 1 0} & \mbox{$y_1 - \sin(\theta_1)$} \\
\mbox{\tt 2 S 1 0} & \mbox{$y_2 - \sin(\theta_2)$} \\
\mbox{\tt 1 C 1 0} & \mbox{$y_3 - \cos(\theta_1)$} \\
\mbox{\tt 2 C 1 0} & \mbox{$y_4 - \cos(\theta_2)$}
\end{array}$$ 
\caption{{\tt inputSystem} and a line-by-line description of the file}\label{Fig:InputSystem}
\end{figure}

\subsubsection*{Points}

Since the number of variables was described in the input system, 
we only need to list the number of points, $2$,
followed by  floating-point representation of the 
real and imaginary parts of the coordinates for each of the points $Z_1$ and $Z_2$.

Figure~\ref{Fig:Points} lists the contents of this file, which we name {\tt points}.

\begin{figure}[!h]
\centering
$$\begin{array}{l}
\mbox{\tt 2} \\
\\
\mbox{\tt 0.711 0} \\
\mbox{\tt 2.261 0} \\
\mbox{\tt 0.65 0} \\
\mbox{\tt 0.77 0} \\
\mbox{\tt 0.76 0} \\
\mbox{\tt -0.64 0} \\
\\
\mbox{\tt 1.874 0} \\
\mbox{\tt 0.324 0} \\
\mbox{\tt 0.95 0} \\
\mbox{\tt 0.32 0} \\
\mbox{\tt -0.30 0} \\
\mbox{\tt 0.95 0}
\end{array}$$
\caption{{\tt points}}\label{Fig:Points}
\end{figure}

\subsubsection*{Configurations}

The last file indicates the settings and algorithms for {\tt alphaCertified} to run.  
For polynomial-exponential systems, we need to utilize floating-point arithmetic
and can set the precision.  

Figure~\ref{Fig:Config} lists the contents of this file, which we name {\tt config},
which instructs {\tt alphaCertified} to use $1024$-bit floating-point arithmetic 
while executing the default certification procedures in {\tt alphaCertified}.
We refer the reader to \cite{alphaCertified} for more details on settings
and algorithms.

\begin{figure}[!h]
\centering
$$\begin{array}{l}
\mbox{\tt ARITHMETICTYPE: 1;} \\
\mbox{\tt PRECISION: 1024;}
\end{array}$$
\caption{{\tt config}}\label{Fig:Config}
\end{figure}

\subsection{Execution}\label{Sec:Execution}

For simplicity, we follow Linux syntax and assume that
a binary file for {\tt alphaCertified} along with the three files
constructed above are in the same folder.  With this setup, we execute
$$\gg \mbox{\tt ./alphaCertified inputSystem points config}$$

\subsection{Output}\label{Sec:Output}

The output of {\tt alphaCertified} is contained in a summary 
of the results printed to the screen, contained in
Figure~\ref{Fig:ScreenOutput}, and several files.  
Figure~\ref{Fig:Summary} contains the portions of the
human readable file {\tt summary} created during the execution of {\tt alphaCertified}.
These values are also printed in the machine readable file
{\tt constantValues}.

\begin{figure}[!h]
\centering
\begin{verbatim}
  alphaCertified v1.3.0 (October 16, 2013)
  Jonathan D. Hauenstein and Frank Sottile
       GMP v4.3.2 & MPFR v3.1.2

Please note that all coefficients must be complex rational numbers.

alphaCertified is using the polynomial-exponential certification algorithms.

Analyzing 2 points using 1024-bit floating point arithmetic.

Isolating 2 approximate solutions.

Classifying 2 distinct approximate solutions.

Floating point (1024 bits) soft certification results:

Number of points tested:                2
Certified approximate solutions:        2
Certified distinct solutions:           2
Certified real distinct solutions:      2
\end{verbatim}
\caption{Summary printed to the screen}\label{Fig:ScreenOutput}
\end{figure}

\begin{figure}[!h]
\centering
$$\begin{array}{lcl}
\mbox{\tt alpha < 1.265465288439055e-1} & & \mbox{\tt alpha < 1.355028294876322e-1} \\
\mbox{\tt beta \mytilde= 4.938677034638513e-3} & & \mbox{\tt beta \mytilde= 5.257805074083256e-3} \\
\mbox{\tt gamma < 2.562356840836994e1} & & \mbox{\tt gamma < 2.577174839659842e1}
 \end{array}$$
\caption{Portions of {\tt summary} corresponding to $Z_1$ and $Z_2$, respectively}\label{Fig:Summary}
\end{figure}

\end{document}